\documentclass[12pt, reqno]{amsart}
\usepackage{amsmath,amssymb,a4}
\theoremstyle{plain}
\newtheorem{theorem}{Theorem}[section]

\newtheorem{corollary}[theorem]{Corollary}

\theoremstyle{definition}
\newtheorem{remark}[theorem]{Remark}
\newtheorem{remarks}[theorem]{Remarks}
\newtheorem{definition}[theorem]{Definition}

\numberwithin{equation}{section}

\newtoks\by
\newtoks\paper
\newtoks\book
\newtoks\jour
\newtoks\yr
\newtoks\pages
\newtoks\vol
\newtoks\publ
\def\ota{{\hbox\vol{???}}}
\def\cLear{\by=\ota\paper=\ota\book=\ota\jour=\ota\yr=\ota
\pages=\ota\vol=\ota\publ=\ota}
\def\endpaper{\the\by, \textit{\the\paper},
{\the\jour} \textbf{\the\vol} (\the\yr), \the\pages.\cLear} 
\def\endbook{\the\by, \textit{\the\book}, \the\publ.\cLear}
\def\endprep{\the\by, \textit{\the\paper}, \the\jour.\cLear}
\def\name#1#2{#1 #2}
\def\et{ and }
\def \n #1.#2{\Vert #1\Vert_{#2}}

\begin{document}

\title{Rubio de Francia's extrapolation Theory: estimates for the distribution function} 
\author {Mar\'ia J.\ Carro$^{*}$,  Javier Soria$^{*}$, and Rodolfo H.\ Torres$^{**}$} 

\address{M.J.\ Carro and J. Soria,   Department of
Applied Mathematics and Analysis, University of Barcelona, 
 08007  Barcelona, Spain}
\email{carro@ub.edu, soria@ub.edu}

\address{R.H. Torres, Department of Mathematics, University of Kansas,  
Lawrence, KS 66045, USA}
\email{torres@math.ku.edu}

\thanks{$^*$Both authors have been  partially supported by Grant MTM2007-60500.}
\thanks{ $^{**}$Author supported in part by National Science Foundation under Grant  DMS 0800492.}

\begin{abstract} 
Let $T$ be an arbitrary operator  bounded from $L^{p_0}(w)$ into $L^{p_0, \infty}(w)$
for every weight $w$ in the Muckenhoupt class $A_{p_0}$. It is proved in this article that the distribution function of $Tf$ with respect to any weight $u$ can be essentially majorized by the distribution function of $Mf$ with respect to $u$ (plus an integral term easy to control). As a consequence,   well-known extrapolation results, including results in a multilinear setting, can be obtained with very simple proofs. New applications in extrapolation for two-weight problems and estimates on rearrangement invariant spaces are established too.
\end{abstract}

\maketitle
\pagestyle{headings}\pagenumbering{arabic}\thispagestyle{plain}

\markboth{Rubio de Francia's extrapolation: estimates for the distribution function}{ }

\section{Introduction}\label{S:intro} 

In 1984, J.L. Rubio de Francia \cite{R} proved that if $T$ is a sublinear operator such that $T$ is bounded on $L^r(w)$ for every $w$ in the Muckenhoupt class $A_r$ ($r>1$) with constant only depending on 
$$
\Vert w\Vert _{A_r}=\sup_Q \bigg({1 \over |Q|} \int_Q w\bigg)\bigg({1\over |Q|} \int_Q w^{-1/(r-1)}\bigg)^{r-1}, 
$$
where the supremum is taken over all cubes $Q$, then for every $1<p<\infty$,  $T$ is bounded on $L^p(w)$ for every $w\in A_p$ with constant only depending on $\Vert w\Vert _{A_p}$. Since then, many interesting papers  concerning this topic have been published (see for example  \cite{CGMP,CMP1,cmp:cmp}). From those results, it is now known that, in fact, the operator $T$ plays no role. That is, if $(f,g)$ are two functions such that for some $p_0\ge 1$, 
\begin{equation}\label{main}
\int_{\mathbb R^n} g^{p_0}(x) w(x) dx \le C \int_{\mathbb R^n} f^{p_0}(x) w(x) dx,
\end{equation}
for every $w\in A_{p_0}$, with $C$ depending on $\Vert w\Vert _{A_{p_0}} $, then for every $1< p<\infty$
$$
\int_{\mathbb R^n} g^{p}(x) w(x) dx \le C  \int_{\mathbb R^n} f^{p}(x) w(x) dx,
$$
for every $w\in A_{p}$ and $C$ depending on $\Vert w\Vert _{A_{p}}$. From here it follows, for example, the so-called weak type version of Rubio de Francia theorem; that is, if $T$ is an operator such that $T$ is bounded from $L^r(w)$ into  $L^{r, \infty}(w)$ for every $w$ in the Muckenhoupt class $A_r$ ($r>1$), with constant only depending on $\Vert w\Vert _{A_r}$,  then $T$ is bounded from $L^p(w)$ into  $L^{p, \infty}(w)$ for every $w\in A_p$, with constant only depending on $\Vert w\Vert _{A_p}$.  

We  want to emphasize here that the classical situation is to extrapolate from a strong-type estimate to a strong-type estimate, while the main applications in this paper are to extrapolate from a weak-type estimate to a weak-type estimate. However, we recover  the strong type boundedness, as in the classical case, by interpolation.

The extrapolation theory has also been generalized to the case of weights in  $A_\infty=\cup_{p>1}A_p$, and many   consequences have been derived from them. 

In fact, the reason why an inequality of the form (\ref{main}) is so interesting is because there are many important operators $T$ in Harmonic Analysis and Partial Differential Equations, such as the  Hardy-Littlewood maximal  operator, singular integrals, commutators, etc., satisfying that 
$$
T:L^p(u) \longrightarrow L^{p}(u)
$$
is bounded for every $u\in A_p$, and hence the pair $(f, Tf)$ satisfies (\ref{main}) for every $f$ and every $u\in A_p$. Also, if $M$ is the Hardy-Littlewood maximal operator and $T$ is a Calder\'on-Zygmund operator, it  is known (see \cite{coif:coif,cf:cf}) that 
$$
\int_{\mathbb R^n} |Tf|^p(x) u(x) dx \le C \int_{\mathbb R^n} |Mf|^p(x) u(x) dx,
$$
for every $u\in A_\infty$, and hence the couple $( Mf, Tf)$ satisfies (\ref{main}) for every $f$ and every $u\in A_\infty$. 

Moreover, since the class of weights in the Muckenhoupt class is so large, an inequality of the form (\ref{main})  contains a lot of information which can be applied to obtain the boundedness  in many other function spaces, as it has been recently shown in \cite{CGMP}. The main results in that paper deal with the boundedness on rearrangement invariant spaces (with or without weights) and since in these spaces one has to measure essentially the level set $\{ x; g(x)>y\}$, it is quite natural to ask if there is some connection between the distribution function of $g$, with respect to any measure, and the function $f$. 

The goal of this paper is to prove this connection.  In fact, as was mentioned before, we start with a weaker condition than (\ref{main}), which is more natural for our purpose, namely 
$$
\Vert g\Vert _{L^{p_0, \infty}(u)}\le \varphi(||u||_{A_{p_0}}) \Vert f\Vert _{L^{p_0}(u)},
$$
for every $u\in A_{p_0}$ (we will also consider the cases $u\in A_{\infty}$ or $u\in A_1$), with $\varphi$ a function locally bounded; that is,  $\varphi$ satisfies that,  for every $M$,  $C_M=\sup_{0<t<M} \varphi (t)<\infty$.

This condition is the standard one in all the results concerning extrapolation and it will be assumed all over the paper (see, for example, \cite{graf}). 

With these estimates (see (\ref{debil1}),  (\ref{dist2}), and  (\ref{dist3})) we can prove the weak type version of Rubio de Francia's extrapolation results  and many others, including  new boundedness properties of operators on different kind of spaces. Also, we can deduce the boundedness of operators for two weights even in the off-diagonal case. 
 We do not pretend to give new proofs of  all the results already known in the literature, but we shall emphasize those that, as far we are concerned, are new or whose proofs are much shorter. In particular, our proof of the classical 
 Rubio de Francia's extrapolation is shorter than previous ones because it is the same proof for every $p$; that is, we do not need to make a difference between the cases  $p<p_0$ and  $p>p_0$.

We shall use the Hardy-Littlewood maximal operator
$$
Mf(x)=\sup_{x\in Q}{1\over |Q|}\int_Q |f(y)| dy,
$$
and, for each $0<\mu<1$, $M_\mu f(x)= \Big(M(f^{1/\mu})(x)\Big)^{\mu}$ (usually, in the literature, this operator is denoted as  $M_{1/\mu}$). The distribution function of $g$, with respect to a  positive locally integrable function $u$, is denoted by  
$$
\lambda_g^u(y)=u(\{ x; g(x)>y\}).
$$
We refer to \cite{BS}   for other definitions and results concerning distribution functions, the decreasing rearrangement function $f^*$ and  rearrangement invariant spaces, and to \cite{gr:gr} for well-known results on  weights and boundedness of operators in weighted Lebesgue spaces. The notation $A \lesssim B$ will denote an inequality of the form $A\le CB$, where the constant $C$ is independent of the fundamental parameters in $A$ and $B$.

\ 

\noindent
{\bf Acknowledgement:} We want to thank Javier Duoandikoetxea for many useful comments and remarks which have  improved the final version of the paper. In particular, he brought to our attention papers \cite{n:n} and \cite{n1:n1}, and we owe him Remark~\ref{duo}. 

\section{Main results on distribution functions}\label{S:main} 

In all the results that follow, we may assume, without loss of generality and by a simple approximation argument, that $g$ is a bounded function with compact support: just take $g_N(x)=g(x)\chi_{\{x\in B(0, N); |g(x)|\le N\}}$ and then make  $N\to\infty$ (Fatou's lemma gives the result). Observe that under these hypothesis:
\begin{equation}\label{finito}
\lambda_g^u(y)<\infty, \quad\text{for all }y>0.
\end{equation}

\subsection{Classical case}

\begin{theorem} \label {teo1}Let $f$ and $g$ be two positive functions such that,  for every $w\in A_{p_0}$ with $1<p_0<\infty$, we have  that
\begin{equation}\label{debil1}
\Vert g\Vert _{L^{p_0, \infty}(w)}\le \varphi(\Vert w\Vert _{p_0}) \Vert f\Vert _{L^{p_0}(w)},
\end{equation}
with $\varphi$ a locally bounded function. Then, for every $0\le\alpha<p_0-1$, every $0<\mu<1$ and every  positive locally integrable function $u$
$$
\lambda^u_g(y)\lesssim  \lambda^u_{Mf}(y) + {1\over y^{p_0-\alpha}} \int_{\mathbb R^n} f^{p_0-\alpha}(x) M_{\mu} (u \chi_{\{ g>y\}})(x) dx.
$$
\end{theorem}

\begin{proof}  For every $\alpha\ge 0$, 
\begin{eqnarray*}
\lambda^u_g(y)&\le& \lambda^u_{Mf}(y)+u(\{x; g(x)>y, \ Mf(x)\le y\}) 
\\
&=& 
 \lambda^u_{Mf}(y)+{1\over y^{p_0}} y^{p_0}\int_{\{ g(x)>y, Mf(x)\le y\}} u(x) dx
 \\
& \le&  \lambda^u_{Mf}(y)+{1\over y^{p_0}} y^{p_0}\int_{\{ g(x)>y \}} \Big ({ y\over Mf(x)}\Big)^\alpha u(x) dx
 \\
& \le &
 \lambda^u_{Mf}(y)+{1\over y^{p_0-\alpha}} y^{p_0}\int_{\{ g(x)>y \}} (Mf(x))^{-\alpha} M_\mu( u \chi_{\{ g>y\}})(x) dx.
 \end{eqnarray*}
 
 Now, if $M_{\mu} (u \chi_{\{ g>y\}})(x)=\infty$ in a set of positive measure, the result is trivial. Otherwise, if $\alpha <p_0-1$,  then $ v(x) =(Mf(x))^{-\alpha} M_\mu( u \chi_{\{ g>y\}})(x)\in A_{p_0}$, with 
$ \Vert v\Vert _{A_{p_0}}\lesssim\bigg( { p_0-1-\alpha}\bigg)^{1-p_0}{1\over 1-\mu}$,
  and hence we can apply the hypothesis to get
 
\begin{eqnarray*}
\lambda^u_g(y)
&\lesssim& \lambda^u_{Mf}(y) + {1\over y^{p_0-\alpha}} \int_{\mathbb R^n} f^{p_0}(x)  (Mf(x))^{-\alpha} M_{\mu} (u \chi_{\{ g>y\}})(x) dx
\\
&\lesssim&
\lambda^u_{Mf}(y) +  {1\over y^{p_0-\alpha}} \int_{\mathbb R^n} f^{p_0-\alpha}(x) M_{\mu} (u \chi_{\{ g>y\}})(x) dx.
\end{eqnarray*}
\end{proof}

The next result will allow us to include the case $p_0=1$ and $\alpha=p_0-1$ in the previous theorem.

\begin{theorem} \label{teo2} Let $f$ and $g$ be two positive functions such that (\ref{debil1})  holds for every $w\in A_{p_0}$, for some  $1\le p_0<\infty$. Then, for every $0<\mu<1$, every $0<\delta<1$, and every positive locally integrable function $u$
\begin{equation}\label{dist2}
\lambda^u_g(y) \lesssim c_{p_0}\lambda^u_{M_{\delta}f}(y) + {1\over y} \int_{\mathbb R^n} f(x) M_{\mu} (u \chi_{\{ g>y\}})(x) dx,
\end{equation}
where
$c_{p_0}=1$ if $p_0>1$ and $0$ if $p_0=1$.
\end{theorem}

\begin{proof} If $p_0>1$  the proof is completely similar to the previous one, except that now we work with $M_{\delta}$ instead of $M$ and hence we can take $\alpha=p_0-1$. 

On the other hand,  if $p_0=1$, then 
\begin{eqnarray*}
\lambda^u_g(y)&=& u(\{x; g(x)>y\}) 
\le 
{y\over y}  \int_{\{ g(x)>y \}} M_\mu( u \chi_{\{ g>y\}})(x) dx\\
&\le& {\varphi\big({C\over 1-\mu}\big)\over y} \int_{\mathbb R^n} f(x) M_{\mu} (u \chi_{\{ g>y\}})(x) dx.
 \end{eqnarray*}
\end{proof}

As mentioned in the introduction, there are also extrapolation results for $A_\infty$ weights. In this case, we have a similar estimate for the distribution function in which the parameter $\alpha$ can be taken up to the value $p_0$.  

\begin{theorem} \label{teoo3} Let $f$ and $g$ be two positive functions such that (\ref{debil1})  holds for every $w\in A_{\infty}$. Then, for every $0\le \alpha\le p_0$, every $0<\mu<1$, every $r>0$, and every  positive locally integrable function $u$, 
\begin{equation}\label{dist3}
\lambda^u_g(y)\lesssim  \lambda^u_{M_r f}(y) + {1\over y^{p_0-\alpha}} \int_{\mathbb R^n} f^{p_0-\alpha}(x) M_\mu(u \chi_{\{ g>y\}})(x) dx.
\end{equation}
\end{theorem}

\begin{proof} In this case  the weight $ v(x) =(M_r f(x))^{-\alpha} M_\mu( u \chi_{\{ g>y\}})(x)\in A_{\infty}$, for every $\alpha\ge 0$, and  
$ \Vert v\Vert _{A_\infty }\lesssim C_{r, \mu, \alpha}$. Hence we get easily the result as in the proof of Theorem \ref{teo1}. 
\end{proof}

The hypothesis of Theorem~\ref{teoo3} is satisfied by the following couples:

\begin{enumerate}
\item[(i)] $(Mf, Tf)$, with $T$ any Calder\'on-Zygmund operator (see \cite{coif:coif,cf:cf}).

\item[(ii)] $(M^2 f, C_{b}f)$, where $M^2$ is the second iteration of $M$, $C_b f(x)=T(bf)(x)- bTf(x)$ is the commutator of  any Calder\'on-Zygmund operator $T$, and $b\in BMO$. Similarly for the higher order commutator $(M^{m+1} f, C_{b}^mf)$  (see \cite{pe3}). 

\item[(iii)] $(M_\alpha f, I_\alpha f)$,  with $M_\alpha$ the fractional maximal operator and $I_\alpha$ the fractional integral, with $0<\alpha<n$ and $n$ is the dimension (see \cite{mw:mw}).
\end{enumerate}

In fact, if the function $f$ in the previous theorem satisfies that, for some $0<\beta<1$, $f^{\beta}\in A_1$, then the term $\lambda^u_{M_rf}(y)$ in (\ref{dist3}) can be substituted by $\lambda^u_{f}(y)$. This observation can be applied to the couple of example (i) to obtain the following result.

\begin{corollary}
 For every  Calder\'on-Zygmund operator $T$,  every $0<\mu<1$ and every $0<q<\infty$, 
\begin{equation*}\label{mwci}
\lambda^u_{Tf}(y)\lesssim  \lambda^u_{Mf}(y) + {1\over y^{q}} \int_{\mathbb R^n} (Mf)^q(x) M_\mu(u \chi_{\{|Tf| >y\}})(x) dx.
\end{equation*}
\end{corollary}

All the previous expressions are very useful to prove the boundedness of $T$ on several function spaces.
Moreover, the fact that no condition is imposed on the function $u$ allows us to obtain boundedness results for two weights and off-diagonal which, as far as we know, are new.

\subsection{Multilinear case}

In this Section, we shall be dealing with extrapolation results for multilinear operators (see \cite{gm:gm}). In fact, the operator $T$ plays no role and hence everything could  be formulated for a triple $(f_1, f_2, g)$. 

We shall give a distribution formula in the case $1\le p_1\le p_2$, but a similar result can be proved if 
$1\le p_2\le p_1$.

\begin{theorem}\label{multi}
Let  $T$ be an operator such that
$$
T:L^{p_1}(w_1)\times L^{p_2}(w_2)\longrightarrow L^{p, \infty}(w)
$$
is bounded, for every $w_1\in A_{p_1}$ and every $w_2\in A_{p_2}$, with  $1\le p_1\le p_2$, 
$$
{1\over p}={1\over p_1}+ {1\over p_2}, \qquad w=w_1^{p/p_1} w_2^{p/p_2}.
$$
Then, for every $0<\mu<1$, every $v\in A_1$,  every $u_1$, $u_2$ and $u=u_1^{\nu_1} u_2^{\nu_2}$,
we have that 
\begin{eqnarray*}\label{distmult}
\lambda^u_{T(f_1, f_2)}(y)& \lesssim& \lambda^u_{M_{\rho}f_1  M_{\rho}f_2}(y)
\\
&\quad &+ {1\over y^{p/p_1}}\bigg[ \bigg(\int_{\mathbb R^n} f_1(x)  M_\mu(u_1^{\beta_1}u_2^{\beta_2}
\chi_{\{ |T(f_1, f_2)|>y   \} } )(x) dx\bigg)^{p/p_1}
\\
&\quad&\quad\times
\bigg(\int_{\mathbb R^n} f_2(x)^{p_2/p_1}  v^{s}M_\mu(v^{-s} u_1^{\gamma_1}u_2^{\gamma_2}\chi_{\{ |T(f_1, f_2)|>y   \} } )(x)dx\bigg)^{p/p_2}\bigg],
\end{eqnarray*}
where  $0<\rho< 1$, 
$
 s=(1-p_2)\Big(1-{p'_2\over p'_1}\Big), 
$
$
\beta_1{p\over p_1}+\gamma_1 {p\over p_2}={\nu_1}, 
$ and
$\beta_2{p\over p_1}+\gamma_2{p\over p_2}={\nu_2}$,  (if $p_1=p_2=1$, then ${p'_2\over p'_1}=1$ and     $s=0$.) 
\end{theorem}

\begin{proof} Let $g=T(f_1, f_2)$. Then, for every $\alpha\ge 0$ 

\begin{eqnarray*}
\lambda^u_{g}(y)&\le& \lambda^u_{M_{\rho}f_1 M_{\rho}f_2}(y)+u(\{x; g(x)>y, \ M_{\rho}f_1(x) M_{\rho}f_2(x)\le y\}) 
\\
&=& 
 \lambda^u_{M_{\rho}f_1 M_{\rho}f_2}(y)+\int_{\{ g(x)>y, M_{\rho}f_1(x) M_{\rho}f_2(x)\le y\}} u(x) dx
 \\
& \le&  \lambda^u_{M_{\rho}f_1M_{\rho}f_2}(y)+\int_{\{ g(x)>y \}} \Big ({ y\over M_{\rho}f_1(x) M_{\rho}f_2(x)}\Big)^\alpha u(x) dx
 \\
& \le &
 \lambda^u_{M_{\rho}f_1 M_{\rho}f_2}(y)
 \\
 &+&{y^{p}\over y^{p-\alpha}} \int_{\{ g>y \}}\frac { v^{{sp\over p_2}} \big(M_\mu( u_1^{\beta_1}u_2^{\beta_2} \chi_{\{ g>y\}})\big)^{p/p_1}\big(
 M_\mu( v^{-s}u_1^{\gamma_1}u_2^{\gamma_2} \chi_{\{ g>y\}})\big)^{p/p_2}}{(M_{\rho}f_1)^{\alpha}(M_{\rho}f_2)^{\alpha}}\, dx.
 \end{eqnarray*}
 
 Now, if we take $\alpha={p\over p'_1}$, we have that  
 $$ 
 v_1(x) =(M_{\rho}f_1(x))^{-\alpha p_1\over p} M_\mu(  u_1^{\beta_1}u_2^{\beta_2} \chi_{\{ g>y\}})(x)\in A_{p_1},
 $$ 
 and also by definition of $s$, 
 \begin{eqnarray*}
 v_2(x) &=&(M_{\rho}f_2(x))^{-\alpha p_2\over p} v(x)^{{s}}M_\mu( v^{-s} u_1^{\gamma_1}u_2^{\gamma_2 }\chi_{\{ g>y\}})(x)
 \\
&=&
 \Big( (M_{\rho}f_2(x))^{-p_2\over p'_1 (1-p_2)} v(x)^{1-\frac{p'_2}{p'_1}}\Big)^{1-p_2 }M_\mu( v^{-s} u_1^{\gamma_1}u_2^{\gamma_2 }\chi_{\{ g>y\}})(x)
  \\
&=&
 \Big[(M_{\rho}f_2(x))^{p'_2\over p'_1 } v(x)^{1-\frac{p'_2}{p'_1}}\Big]^{1-p_2 }M_\mu( v^{-s} u_1^{\gamma_1}u_2^{\gamma_2 }\chi_{\{ g>y\}})(x)\in A_{p_2},
\end{eqnarray*}
since clearly the weight between brackets is in $A_1$. Hence,   we can apply the hypothesis to get
 
\begin{eqnarray*}
\lambda^u_g(y)\!
&\le&\! \lambda^u_{M_{\rho}f_1 M_{\rho}f_2}(y) 
\\
&\quad&+
 {1\over y^{p-\alpha}} \bigg[ \bigg( \int_{\mathbb R^n} f_1^{p_1}(x)  (M_{\rho}f_1(x))^{1-p_1} M_\mu(  u_1^{\beta_1}u_2^{\beta_2} \chi_{\{ g>y\}})(x) dx\bigg)^{p/p_1}
\\
&\quad&\quad\times
\bigg( \int_{\mathbb R^n} f_2^{p_2}(x)  (M_{\rho}f_2(x))^{- p_2\over p'_1} v^{{s}}(x) M_\mu(  v^{-s}u_1^{\gamma_1}u_2^{\gamma_2} \chi_{\{ g>y\}})(x) dx\bigg)^{p/p_2}\bigg],
\end{eqnarray*}
from which the result follows. 
\end{proof}

 \subsection{Two weights case} 
 
 Concerning the boundedness of the Hardy-Littlewood maximal operator with two weights, it is known that, if $p\ge 1$, then
 $$
 M:L^{p}(u)\longrightarrow L^{p,\infty}(v),
 $$
 if and only if $(u,v)\in A_p$; that is,
\begin{equation*}\label{app}
\Vert (u,v)\Vert _{A_p}=\bigg( {1\over |Q|} \int_Q v \bigg)\bigg( {1\over |Q|}\int_Q u^{1-p'}\bigg)^{p-1}<\infty,
 \end{equation*}
and that this condition is not sufficient for the strong boundedness. 
 
 In fact (see \cite{Duo}), it is very easy to see that the couple
\begin{equation}\label{label}
 ((Mw_0) w_1^{1-p}, w_0 (Mw_1)^{1-p})\in A_p,
 \end{equation}
 for every locally integrable functions $w_0$ and $w_1$, and hence if the $A_p$ condition were sufficient for the strong boundedness, we would have that, for every  locally integrable $w$
 $$
 \int_{\mathbb R^n} Mf(x)^p (Mw)^{1-p}(x)  (x) \le  \int_{\mathbb R^n} f(x)^p  w^{1-p} (x) dx,
 $$
 which, taking $w=f$, would imply that  $M:L^1\longrightarrow L^1$, which is a contradiction. 
 
 Observe that by changing, for example, $u$ by $u/\Vert (u,v)\Vert _{A_p}$ we can always assume that $\Vert (u,v)\Vert _{A_p}=1$. 
 
 In \cite{n:n}, some extrapolation results were proved for weights $(u,v)\in A_p$. The distribution formula in this case is the following. 
 
 \begin{theorem} \label{teo2pesos} Let $f$ and $g$ be two positive functions such that, for some  $1\le p_0<\infty$ and  for every 
 $(u_1,u_2)\in {A_{p_0}}$, with $\Vert (u_1,u_2)\Vert _{A_{p_0}}=1$, 
 $$
 \Vert g\Vert _{L^{p_0, \infty}(u_2)}\le  \Vert f\Vert _{L^{p_0}(u_1)}.
 $$
Then, for every  positive locally integrable function $u$, 
\begin{equation}\label{dist5}
\lambda^u_g(y)\le c_{p_0}\lambda^u_{Mf}(y) + {1\over y} \int_{\mathbb R^n} f(x) M (u \chi_{\{ g>y\}})(x) dx,
\end{equation}
and $c_{p_0}=1$ if $p_0>1$ and 0 if $p_0=1$.
\end{theorem}

\begin{proof} Using (\ref{label}) we have that if $p_0>1$, 
\begin{eqnarray*}
\lambda^u_g(y)&\le & u(\{x; Mf(x)>y\}) + u(\{x; g(x)>y, Mf(x)\le y\})
\\
&\le & 
\lambda_{Mf}^u(y)+  \int_{\{ g(x)>y \}} \bigg({y\over Mf(x)}\bigg)^{p_0-1}u \chi_{\{ g>y\}}(x) dx
\\
&\le&
\lambda_{Mf}^u(y)+{y^{p_0}\over y}  \int_{\{ g(x)>y \}} Mf(x)^{1-p_0}u \chi_{\{ g>y\}}(x) dx
\\
&\le&
\lambda_{Mf}^u(y)+{1\over y}  \int_{\mathbb R^n} f^{p_0}(x) f(x)^{1-p_0}M(u \chi_{\{ g>y\}})(x) dx
\\
&=&
\lambda_{Mf}^u(y)+{1\over y}  \int_{\mathbb R^n} f(x) M(u \chi_{\{ g>y\}})(x) dx. 
 \end{eqnarray*}
  The case $p_0=1$ is completely similar. 
\end{proof}

\section{Applications}\label{S:Appl} 

We use the estimates shown in Section~\ref{S:main} to give 
very direct proofs of some already known results, including the (weak type) Rubio de 
Francia's extrapolation theorem, and to prove new ones in the setting of two 
weights inequalities.

\subsection{Rubio de Francia's extrapolation results:}

\ 

\ 

\centerline{\bf Extrapolation for $A_{p}$ weights}

\

\begin{theorem}\label{rf1} Let $f$ and $g$ be two positive functions such that (\ref{debil1})  holds for some  $1\le p_0<\infty$ and  for every $w\in A_{p_0}$. Then, for every $1<p<\infty$ and   every $w\in A_{p}$,  \begin{equation}
\label{pp}
\Vert g\Vert _{L^{p, \infty}(w)}\le C_w \Vert f\Vert _{L^{p}(w)}. 
\end{equation}

\end{theorem}

\begin{proof} Let us prove (\ref{pp}) by using Theorem \ref{teo2}. By (\ref{dist2}) we have, for  $u=w\in A_p$, 
\begin{eqnarray*}
\lambda^w_g(y)&\lesssim&  \lambda^w_{M_\delta f}(y) + {1\over y} \int_{\mathbb R^n} f(x) M_{\mu} (w \chi_{\{ g>y\}})(x) dx
\\
&\le&
 \lambda^w_{M_\delta f}(y) 
+ {1\over y} \bigg( \int_{\mathbb R^n} f^p w dx\bigg)^{1/p} 
 \bigg( \int_{\mathbb R^n} M_{\mu} ( w\chi_{\{ g>y\}})^{p'} w^{1-p'}dx\bigg)^{1/p'}.
\end{eqnarray*}
Now, if $w\in A_p$,  $w^{1-p'}\in A_{p'}$ and hence, we can take $0<\mu<1$ so that  $w^{1-p'}\in A_{\mu p'}$. Therefore, we can estimate the last term by
\begin{eqnarray*}
& &\bigg( \int_{\mathbb R^n} M_{\mu} ( w\chi_{\{ g>y\}})^{p'} w^{1-p'}dx\bigg)^{1/p'}
\lesssim \bigg( \int_{{\{ g>y\}}} w (x) dx\bigg)^{1/p'},
\end{eqnarray*}
and thus, since there exists $0<\delta <1$ such that $w\in A_{\delta p}$, 
\begin{eqnarray*}
& & \lambda^w_g(y)\lesssim {1\over y^p}  \Vert f\Vert _{L^p(w)}^p +{1\over y^{p}}  \Vert f\Vert _{L^p(w)}\bigg(y^p \int_{{\{ g>y\}}} w (x) dx\bigg)^{1/p'}.
\end{eqnarray*}
Then, 
\begin{eqnarray*}
 y^p \lambda^w_g(y)&\lesssim& \Vert f\Vert _{L^p(w)}^p  +  \Vert f\Vert _{L^p(w)}\bigg(y^p  \lambda^w_g(y)\bigg)^{1/p'},
\end{eqnarray*}
from which it  follows, recalling that we may assume $g$ bounded with compact support and using  (\ref{finito}),   that
$ y  \lambda^w_g(y)^{1/p}\lesssim  \Vert f\Vert _{L^p(w)}$ as we wanted to see. 
\end{proof}

From Theorems \ref{teo1} and \ref{rf1}, we have the following.

\begin{corollary} \label{conclu} Let $f$ and $g$ satisfy (\ref{debil1})  for every $w\in A_{p_0}$, with $1< p_0<\infty$. 
 Then, for every $s>1$,  every $0<\mu<1$ and every  locally integrable positive function $u$
\begin{equation}\label{distF}
\lambda^u_g(y)\lesssim  \lambda^u_{Mf}(y) + {1\over y^{s}} \int_{\mathbb R^n} f^{s}(x) M_{\mu} (u \chi_{\{ g>y\}})(x) dx.
\end{equation}

\end{corollary}

We can also obtain a new boundedness result for two weights even in the off-diagonal case. 

\begin{definition} Given $0<p,q<\infty$, we  say that a pair of weights $(u,v)\in A_{p, q}$ if 
$$
M: L^p(u)\longrightarrow L^{q, \infty} (v),
$$
and we say that 
$(u,v)\in S_{p, q}$ if 
$$
M: L^p(u)\longrightarrow L^{q} (v). 
$$
If $p=q$ we will write $S_p=S_{p,p}$.
\end{definition}

\begin{remarks}
  Regarding conditions $A_{p, q}$ and $S_{p,q}$, the following facts are known (see \cite{CRS}):

\begin{enumerate}
\item[(a)]  If $(u,v)\in A_{p, q}$, then
\begin{enumerate}
\item[{(a.1)}] $u\notin L^1$.

\noindent
\item[{(a.2)}] $p\ge 1$.

\noindent
\item[{(a.3)}] If $u=v$, then $p=q$. 

\noindent
\item[{(a.4)}]
$
\Vert u^{-1} \chi_Q\Vert _{L^{p'}(u)}\Vert  \chi_Q\Vert _{L^{q}(v)}\lesssim |Q|.
$

\noindent
\item[{(a.5)}] The case  $p=1$ and $q<1$ was characterized by Lai (see \cite{lai}).
\end{enumerate}
\noindent
\item[{(b)}] 
If $p<1$ and $0<q<\infty$, then $A_{p, q}=\emptyset$.
\noindent
\item[{(c)}] 
The condition $(v^{1-q'}, u^{1-p'})\in S_{\mu q', \mu p'}$ was characterized by Sawyer in the range $1<p\le q<\infty$ and $\mu q'>1$ (see \cite{s:s}). 
\item[{(d)}] Other sufficient conditions for $(u,v)\in S_p$  can be found in \cite{pe4}.

\end{enumerate}
\end{remarks}

\begin{theorem} \label{rfdp} Let $f$ and $g$ be two positive functions such that (\ref{debil1})  holds for some  $1\le p_0<\infty$ and every $w\in A_{p_0}$,  and let $1< p, q<\infty$.  If
 $(u,v)\in A_{\delta p, \delta q}$ for some $0<\delta<1$ and $(v^{1-q'}, u^{1-p'})\in S_{\mu q', \mu p'}$ for some  $0<\mu<1$, 
then
$$
\Vert g\Vert _{L^{q,\infty}(v)} \lesssim \Vert f\Vert _{L^p(u)}.
$$
\end{theorem}

\begin{proof}
Since $(u,v)\in A_{\delta p, \delta q}$, it holds that 
$$
\lambda^v_{M_{\delta}f}(y)\lesssim {\Vert f\Vert _{L^p(u)}^q \over y^q},
$$
and since $(v^{1-q'}, u^{1-p'})\in S_{\mu q', \mu p'}$
\begin{eqnarray*}
 \int_{\mathbb R^n} f(x) M_{\mu} (v \chi_{\{ g>y\}})(x) dx&\lesssim& \Vert f\Vert _{L^p(u)} \bigg(\int_{\mathbb R^n} M_{\mu} (v \chi_{\{ g>y\}})^{p'} u^{1-p'}(x) dx \bigg)^{1/p'}
 \\
 &\lesssim&  \Vert f\Vert _{L^p(u)} \lambda^v_g(y)^{1/q'}.
\end{eqnarray*}

Then  using (\ref{dist2}) we get the result. 
\end{proof}

\begin{remark} \label{duo}
In the case $p=q$ the previous result is trivial since, under the assumed hypothesis it can be proved that there exists 
$s>1$ such that $M:L^p(u^s)\rightarrow L^p(v^s)$ and hence (see \cite{n1:n1,duo:duo}) there exists $w\in A_p$ such that $v\lesssim w\lesssim u$. 

\end{remark}

\begin{theorem} \label{rfdp2} Let $f$ and $g$ be two positive functions such that (\ref{debil1})  holds for some  $1\le p_0<\infty$ and every $w\in A_{p_0}$, and let  $q< 1<p$ and $v\in L^1$. If 

\begin{enumerate}
\item[{(a)}] $(u,v)\in A_{\delta p, \delta q}$, for some $0<\delta<1$, and $M_\mu(v)\in L^{p'}(u^{1-p'})$, for some $0<\mu<1$, 
\end{enumerate}
or 

\begin{enumerate}
\item[{(b)}] $(u,v)\in A_{ p,  q}$ and there exists $1<s  <p$ such that  $M_\mu(v)\in L^{{p\over p-s}}(u^{1-{p\over p-s}})$, for some $0<\mu<1$, 
\end{enumerate}
then 
$$
\Vert g\Vert _{L^{q,\infty}(v)} \lesssim \Vert f\Vert _{L^p(u)}.
$$

\end{theorem}

\begin{proof} We only prove (a), since (b) follows similarly using Corollary \ref{conclu}. As before, we have that 
\begin{eqnarray*}
 \int_{\mathbb R^n} f(x) M_{\mu} (v \chi_{\{ g>y\}})(x) dx&\lesssim& \Vert f\Vert _{L^p(u)} \bigg(\int_{\mathbb R^n} M_{\mu} (v \chi_{\{ g>y\}})^{p'} u^{1-p'}(x) dx \bigg)^{1/p'}.
 \end{eqnarray*}
 Now since $q< 1$, we have by hypothesis that
 \begin{align*}
 \sup_{E}&\bigg(\int_{\mathbb R^n} M_{\mu} (v \chi_{E})^{p'} u^{1-p'}(x) dx \bigg)^{1/p'}v(E)^{{1-q\over q}}
\\
\le&\bigg(\int_{\mathbb R^n} M_{\mu} (v )^{p'} u^{1-p'}(x) dx \bigg)^{1/p'}||v||_1^{{1-q\over q}}<\infty
\end{align*}
 and hence
 $$
 \bigg(\int_{\mathbb R^n} M_{\mu} (v \chi_{\{ g>y\}})^{p'} u^{1-p'}(x) dx \bigg)^{1/p'}\lesssim v(\{ g>y\})^{{q-1\over q}}=\lambda^v_g(y)^{{q-1\over q}},
 $$
and the result follows as in Theorem~\ref{rfdp}. 
\end{proof}

\bigbreak

\begin{remarks}
\
\begin{enumerate}
\item[{(a)}] One can immediately see from the proof that the previous theorem holds for $q=1$ without the hypothesis  $v\in L^1$. 

\item[{(b)}] Observe that Theorem \ref{rf1}  is a particular case of Theorem \ref{rfdp}, since if $u=v\in A_p$, the hypotheses of that theorem hold. 

\end{enumerate}
\end{remarks}

If we consider the case when $f=\chi_E$, then something more can be said since it is known (see \cite{CRS}) that 
$$
\Vert M \chi_E \Vert _{L^{q, \infty}(v)}\le \Vert \chi_E\Vert _{L^{p}(u)},
$$
for every measurable set $E$ if and only if 
\begin{equation}\label{JA}
{v^{1/q}(\cup_j Q_j)\over u^{1/p}(\cup_j E_j) }\lesssim \max {|Q_j|\over |E_j|},
\end{equation}
for every finite collection of pairwise disjoint cubes $Q_j$ and measurable sets $E_j\subset Q_j$. We shall denote condition (\ref{JA}) by $(u,v)\in R_{p,q}$.

\begin{theorem} \label {restricted} Let $T$ be an operator such that
$$
T:L^{p_0}(w) \longrightarrow L^{p_0, \infty}(w),
$$
for every $w\in A_{p_0}$. Then, for every $p> 1$,  
$$
\Vert T \chi_F \Vert _{L^{q, \infty}(v)}\le \Vert \chi_F\Vert _{L^{p}(u)},
$$
for every measurable set $F$,  if the following conditions hold: 

\begin{enumerate}
\item[{(a)}] Case $1\le  q<\infty$:  $(u,v)$ satisfies (\ref{JA}) and for some $0<\mu<1$,  some $s>1$ and every measurable sets $E$ and $F$, 
$$
\int_F M_\mu(v\chi_E)(x)dx \lesssim v(E)^{1-{s\over q}}u(F)^{{s\over p}},
$$
or  $(u,v)\in R_{\delta p, \delta q}$ for some $0<\delta<1$  and for some $0<\mu<1$ and every measurable sets $E$ and $F$, 
$$
\int_F M_\mu(v\chi_E)(x)dx \lesssim v(E)^{{1\over q'}}u(F)^{{1\over p}}.
$$

\item[{(b)}] Case $q\le 1$: $v\in L^1$   and either $(u,v)\in R_{ p,  q}$ and  for some $s>1$
$$
\int_F M_\mu(v)(x)dx \lesssim  u(F)^{{s\over p}},
$$
or $(u,v)\in R_{\delta p, \delta q}$ for some $0<\delta<1$ and  
$$
\int_F M_\mu(v)(x)dx \lesssim  u(F)^{{1\over p}}. 
$$
\end{enumerate}
\end{theorem}

\begin{proof} Using formula (\ref{distF}) with $f=\chi_F$, we get
$$
\lambda_{T\chi_F}^v(y)\lesssim \lambda_{M\chi_F}^v(y)+{1\over y^s}\int_F M_\mu(v\chi_{\{|T\chi_F|>y\}})(x) dx.
$$

Now to prove the first part of (a) we see that by (\ref{JA}), 
$$
\lambda_{M\chi_F}^v(y)\lesssim {1\over y^q} u(F)^{q/p}
$$
and by hypothesis
$$
\int_F M_\mu(v\chi_{\{|T\chi_F|>y\}})(x) dx\lesssim \lambda_{T\chi_F}^v(y)^{1-s/q} u(F)^{s/p}.
$$
Therefore, 
$$
\lambda_{T\chi_F}^v(y)\lesssim {1\over y^q} u(F)^{q/p}
+ {1\over y^s}\lambda_{T\chi_F}^v(y)^{1-s/q} u(F)^{s/p},
$$
from which the result follows. 

For the second part of (a) we proceed as before but we now use formula (\ref{dist2}). 

To prove (b) we proceed as in a) and we use that since $v\in L^1$ and $q<s$, 
$$
\lambda_{T\chi_F}^v(y)^{1/q}\lesssim \lambda_{T\chi_F}^v (y)^{1/s}
$$
for every $s\ge 1$. 

\end{proof}

\centerline{\bf Extrapolation for $A_\infty$ weights}

\ 

Our next theorem is the weak type extrapolation result for weights in $A_\infty$, which can be found in \cite{cmp:cmp}. 

\begin{theorem}\label{nola} Let $f$ and $g$ be two positive functions such that (\ref{debil1})  holds for every $w\in A_{\infty}$. Then, for every $0<p<\infty$ and every $w\in A_{\infty}$ we have that 
$$
\Vert g\Vert _{L^{p, \infty}(w)}\le C_w \Vert f\Vert _{L^{p}(w)}.
$$

\end{theorem}

\begin{proof} The proof follows using (\ref{dist3}) with $r$ such that $w\in A_{rp}$.

\end{proof}

From this result we can conclude that under the hypothesis of Theorem~\ref{nola} we have that, for every $s, r>0$, 
\begin{equation}\label{*}
\lambda^v_g(y)\lesssim  \lambda^v_{M_rf}(y) + {1\over y^{s}} \int_{\mathbb R^n} f^{s}(x) M_\mu(v \chi_{\{ g>y\}})(x) dx,
\end{equation}\
 and hence:

\begin{theorem}  Let $f$ and $g$ be two positive functions such that (\ref{debil1})  holds for every $w\in A_{\infty}$, for some  $1\le p_0<\infty$. Then, for every $1\le p<\infty$ we have that 
$$
\Vert g\Vert _{L^{q,\infty}(v)} \lesssim \Vert f\Vert _{L^p(u)},
$$
if $(u,v)\in A_{p, q}$  and one the following condition  holds: 

\begin{enumerate}
\item[{(a)}] there exists $0<s<\min ({p,q})$ such that $(v^{1-{q\over q-s}}, u^{1-{p\over p-s}})\in S_{\mu {q\over q-s}, \mu {p\over p-s}}$, for some  $0<\mu<1$,  

\item[{(b)}]$q\le 1$  and  $M_\mu(v)\in L^{{p\over p-q}}(u^{1-{p\over p-q}})$, for some $0<\mu<1$. 
\end{enumerate}
\end{theorem}

\begin{proof}  (a) follows as usual using (\ref{*}) with $r=1$ and (b)  with $r=1$ and $s=q$.
\end{proof}
 
\medskip

\centerline{\bf Extrapolation for multilinear operators}

\ 

We give as an application the weak extrapolation result proved in \cite{gm:gm}. In Section \ref{strong} we shall give the strong-type version.

\begin{theorem}\label {rfml} Let  $T$ be an operator such that, for some  $p_1, p_2\ge 1$ and 
$$
{1\over p}={1\over p_1}+ {1\over p_2}
$$
we have that 
$$
T:L^{p_1}(w_1)\times L^{p_2}(w_2)\longrightarrow L^{p, \infty}(w)
$$
is bounded for every $w_1\in A_{p_1}$ and  every $w_2\in A_{p_2}$, where $w=w_1^{p/p_1} w_2^{p/p_2}$. Then, for every $q_1, q_2> 1$, 
$$
{1\over q}={1\over q_1}+ {1\over q_2}, \qquad w=w_1^{q/q_1} w_2^{q/q_2}
$$
with $w_1\in A_{q_1}$ and every $w_2\in A_{q_2}$ we have that 
$$
T:L^{q_1}(w_1)\times L^{q_2}(w_2)\longrightarrow L^{q,\infty}(w)
$$
is bounded.
\end{theorem}

\begin{proof}  We can assume, without loss of generality, that  $p_1\le p_2$ and also by truncating  $T(f_1, f_2)$ if necessary we can assume that, for every $y>0$, 
$$
\lambda_{T(f_1, f_2)}^w(y)<\infty.
$$

The proof follows two steps. First, we see that  the result is true for $q_1=q_2=p_2$ and then we extrapolate from this diagonal point to any other. We shall use the distribution formula of Theorem \ref{multi}. 

Let $w_1, w_2\in A_{p_2}$ and $w=w_1^{1/2}w_2^{1/2}$. Let us  start by estimating 
\begin{eqnarray*}
& &\int_{\mathbb R^n} f_1(x) M_\mu(w_1^{\beta_1}w_2^{\beta_2}
\chi_{\{ |T(f_1, f_2)|>y   \} } )(x) dx\\
&\le& \bigg( \int_{\mathbb R^n} f_1^{p_2} w_1 dx\bigg)^{1/p_2} 
 \bigg( \int_{\mathbb R^n} M_\mu(w_1^{\beta_1}w_2^{\beta_2}
\chi_{\{ |T(f_1, f_2)|>y   \} } )(x)^{p'_2}  w_1^{1-p'_2}dx\bigg)^{1/p'_2}.
\end{eqnarray*}
Now, if $w_1\in A_{p_2}$, then $w_1^{1-p'_2}\in A_{p'_2}$ and then $w_1^{1-p'_2}\in A_{\mu p'_2}$, for some $0<\mu<1$. 
Therefore, we can estimate the last term by
\begin{eqnarray*}
 \bigg( \int_{{\{ T(f_1, f_2)>y\}}} w_1 (x)^{\beta_1 p'_2+1- p'_2} w_2(x)^{\beta_2  p'_2} dx\bigg)^{1/ p'_2}\end{eqnarray*}
 and hence we have to choose the parameters involved in such a way that
\begin{equation}\label{EQ1}
\beta_1 p'_2+1-p'_2={1\over 2}, \qquad {\beta_2 p'_2} ={1\over 2}.
\end{equation}

Now,  we deal with the second term in Theorem \ref{multi}, 
\begin{eqnarray*}
& &\int_{\mathbb R^n} f_2(x)^{p_2/p_1}  v(x)^{s} M_\mu(v^{-s} w_1^{\gamma_1}w_2^{\gamma_2}
\chi_{\{ |T(f_1, f_2)|>y   \} } )(x) dx
\\
&\le&
\bigg(\int_{\mathbb R^n} f_2(x)^{p_2} w_2(x) dx\bigg)^{1/p_1}
\\
&\quad&\quad\times
\bigg(\int_{\mathbb R^n} w_2(x)^{1-p'_1} v(x)^{sp'_1} M_\mu(v^{-s} w_1^{\gamma_1}w_2^{\gamma_2}
\chi_{\{ |T(f_1, f_2)|>y   \} } )(x)^{p'_1} dx\bigg)^{1/p'_1}.
\end{eqnarray*}

Since $w_2\in A_{p_2}$, we have that there exist two weights in $A_1$ such that $w_2= w_{2,1}^{1-p_2} w_{2,2}$. Let us take $v= w_{2,1}$. Then, by definition of $s$ in Theorem \ref{multi}, 
$$
w_2(x)^{1-p'_1} v(x)^{sp'_1}= w_{2,1}(x)^{(1-p_2)(1-p'_1)}w_{2,2}(x)^{1-p'_1}w_{2,1}(x)^{ sp'_1}=
w_{2,2}(x)^{1-p'_1} w_{2,1}(x),
$$
and consequently $w_2(x)^{1-p'_1} v(x)^{sp'_1} \in A_{p'_1}$. Therefore, we can estimate the last term by 
$$
\bigg(\int_{\{ |T(f_1, f_2)|>y   \}} w_2(x)^{1-p'_1}  w_1^{p'_1\gamma_1}w_2^{p'_1 \gamma_2}
 dx\bigg)^{1/p'_1}.
$$
Hence,  we have to choose 
\begin{equation}\label{EQ2}
\gamma_2 p'_1+1-p'_1={1\over 2}, \qquad {\gamma_1 p'_1} ={1\over 2}. 
\end{equation}
Since equations (\ref{EQ1}) and (\ref{EQ2}) are compatible with the fact that 
$$
\beta_j {p\over p_1}+\gamma_j {p\over p_2}={1\over 2}, 
$$
we obtain,  by Theorem \ref{multi},  that 
\begin{eqnarray*}
\lambda^w_{T(f_1, f_2)}(y)& \lesssim& \lambda^w_{M_{\rho}f_1 M_{\rho}f_2}(y)
\\
&\quad&+ {1\over y^{p/p_1}}\bigg( \int_{\mathbb R^n} f_1^{p_2} w_1 dx\bigg)^{{p\over p_1p_2}}  \bigg( \int_{\mathbb R^n} f_2^{p_2} w_2 dx\bigg)^{{p\over p_1p_2}} 
\lambda^w_{T(f_1, f_2)}(y)^{1-{2p\over p_1p_2}},
\end{eqnarray*}
from which we easily obtain, taking $\rho$ sufficiently near 1 to have that $w_1\in A_{\rho p_1}$ and $w_2\in A_{\rho p_2}$,  that
$$
T:L^{p_2}(w_1)\times L^{p_2}(w_2)\longrightarrow L^{p_2/2, \infty}(w),
$$
is bounded and the first step is finished.

Now, we have to extrapolate from the diagonal point $(p_2, p_2)$ to any other point. In this case, we proceed exactly as before. In fact, this case is easier since  the parameter $s$ in Theorem \ref{multi} is zero. 

Taking now

\begin{equation*}
\beta_1 q_1'+1-q'_1={q\over q_1}, \qquad {\beta_2 q'_1} ={q\over q_2}
\end{equation*}
\begin{equation*}
\gamma_2 q_2'+1-q'_2={q\over q_2}, \qquad {\gamma_1 q'_2} ={q\over q_1},
\end{equation*}
we only have to see that these equations are compatible  with the fact that, in this case,  
$$
\beta_j+\gamma_j=2{q\over q_j},
$$
which is easy to see. 
\end{proof}

\begin{remark}
From the above result and applying Theorem \ref{multi}, we can deduce (taking $p_1=p_2=2$) that under the hypothesis of Theorem~\ref{rfml}, we have that, for every $0<\mu<1$,   every $u_1$, $u_2$ and $u=u_1^{\nu_1} u_2^{\nu_2}$:
\begin{eqnarray}\notag
\lambda^u_{T(f_1, f_2)}(y)& \lesssim& \lambda^u_{M_{\rho}f_1  M_{\rho}f_2}(y)
\\ \label{distmult2}
&\quad & +{1\over y^{1/2}}\bigg[ \bigg(\int_{\mathbb R^n} f_1(x)  M_\mu(u_1^{\beta_1}u_2^{\beta_2}
\chi_{\{ |T(f_1, f_2)|>y   \} } )(x) dx\bigg)^{1/2}
\\ \notag
&\quad&\quad\times
\bigg(\int_{\mathbb R^n} f_2(x) M_\mu(u_1^{\gamma_1}u_2^{\gamma_2}\chi_{\{ |T(f_1, f_2)|>y   \} } )(x)dx\bigg)^{1/2}\bigg],
\end{eqnarray}
where  $0<\rho< 1$, 
$$
{\beta_1 \over 2}+{\gamma_1  \over  2}={\nu_1}, 
\qquad\qquad
{\beta_2 \over 2}+{\gamma_2 \over 2}={\nu_2}.
$$
\end{remark} 

This formula is quite useful to obtain  new results concerning the three weights problem for multilinear operators. We omit the proof since it follows the standard technique already developed in the linear case.

\begin{definition} We say that a triple of weights $(u_1, u_2;v)\in A_{p_1, p_2}$ if 
the operator $M(f_1, f_2)= M(f_1)M( f_2)$
satisfies
$$
M: L^{p_1}(u_1)\times L^{p_2}(u_2)\longrightarrow L^{p, \infty} (v),
$$
where $1/p=1/p_1+1/p_2.$
\end{definition}

It is easy to see that, for example, if $(u_j,v)\in A_{p_j}$, then $(u_1,u_2;v)\in A_{p_1,p_2}$.

\begin{theorem}  Let $T$ be a multilinear operator satisfying the hypothesis of Theorem \ref{rfml} and let $q_1, q_2>1$. 
If  $(u_1, u_2;v)\in A_{\rho q_1, \rho q_2}$, for some $0<\rho<1$, and $(v^{1-q_j'}, u_j^{1-q_j'})\in S_{\mu q_j'}$, for some 
$0<\mu<1$ and $j=1,2$,  then
$$
\Vert T(f_1, f_2)\Vert _{L^{q,\infty}(v)} \lesssim \Vert f_1\Vert _{L^{q_1}(u_1)}\Vert f_2\Vert _{L^{q_2}(u_2)},
$$
where $1/q=1/q_1+1/q_2.$
\end{theorem}

\begin{proof} Taking in formula (\ref{distmult2}), $u_1=u_2=v=u$, $\nu_j={p\over p_j}$, we have that $\beta_1=\gamma_2=1$, $\beta_2=\gamma_1=0$ and obtain:
\begin{eqnarray}\notag
\lambda^v_{T(f_1, f_2)}(y)& \lesssim& \lambda^v_{M_{\rho}f_1  M_{\rho}f_2}(y)
\\ \notag
&\quad & +{1\over y^{1/2}}\bigg[ \bigg(\int_{\mathbb R^n} f_1(x)  M_\mu(v
\chi_{\{ |T(f_1, f_2)|>y   \} } )(x) dx\bigg)^{1/2}
\\ \notag
&\quad&\quad\times
\bigg(\int_{\mathbb R^n} f_2(x) M_\mu(v\chi_{\{ |T(f_1, f_2)|>y   \} } )(x)dx\bigg)^{1/2}\bigg],
\end{eqnarray}
and we now  
proceed as in the linear case (see Theorem \ref{rfdp}). 
\end{proof}

\bigbreak
 
 \centerline{\bf Extrapolation with two weights}

\ 

Using (\ref{dist5}) we can prove the following result for two weights which also holds in the non-diagonal case, (see \cite{n:n} for several results in  the diagonal case $p=q$). 

\begin{theorem} Let $f$ and $g$ be two positive functions satisfying the hypothesis of Theorem \ref{teo2pesos}.  Then, for every $1<p, q<\infty$ and every $(u,v)\in A_{p,q}$ such  that $(v^{1-q'}, u^{1-p'})\in S_{q', p'}$, we have that 
$$
\Vert g\Vert _{L^{q, \infty}(v)}\le C  \Vert f \Vert _{L^p(u)},
$$
with  $C$ depending on $\Vert (u,v)\Vert _{A_p}$.

\end{theorem}

\begin{proof} Using (\ref{dist5}), if $(u, v)\in A_{p,q}$, 
\begin{eqnarray*}
 \lambda^v_g(y)&\le&  \lambda^v_{Mf}(y) + {1\over y } \int_{\mathbb R^n} f(x) M (v \chi_{\{ g>y\}})(x) dx
\\
&\le&
\bigg({\Vert f\Vert _{L^p(u)}\over y}\bigg)^q
 + {1\over y}  \Vert f \Vert _{L^{p}(u)} \Vert M( v\chi_{\{ g>y\}})\Vert _{L^{p'}(u^{1-p'})}
 \\
&\le&
\bigg({\Vert f\Vert _{L^p(u)}\over y}\bigg)^q
 + {1\over y}  \Vert f \Vert _{L^{p}(u)} \lambda^v_g(y)^{1/q'},
\end{eqnarray*}
from which the result follows.
\end{proof}

 \subsection{Boundedness of operators defined in Rearrangement Invariant Banach Fun\-ction Spaces (RIBFS)  with weights}

 \ 
 
 \ 
  
 The results  in this subsection are closely related to those proved in \cite{CGMP}. Our main contribution is that we will only assume that (\ref{debil1}) holds for every $w\in A_{p_0}$, which should be compared with the hypothesis made in  \cite{CGMP}, where the couple of functions $(f,g)$ satisfies (\ref{main}) for every $w\in A_\infty$.

 Some standard definitions and notations on an rearrangement invariant space  $X$ (r.i.\ from now on),  that we  will need are the following (see \cite{BS}): $\bar X$ is the representation space such that $\Vert f\Vert _X=\Vert f^*\Vert _{\bar X}$,  $\varphi_X(t)=\Vert \chi_E\Vert _X$ is the fundamental function of $X$, where $E$ is any measurable set such that  $|E|=t$. $X'$ is the associate space defined by
$$
\Vert f\Vert _{X'}=\sup_{\Vert g\Vert _X\le 1} \bigg| \int_{\mathbb R^n} f(x) g(x) dx \bigg|. 
$$
Also, $p_X$ and $q_X$ are the lower and upper Boyd-indices defined by 
$$
p_X=\lim_{t\to\infty}{\log t\over \log h_X(t)},\qquad q_X=\lim_{t\to 0^+}{\log t\over \log h_X(t)},
$$
where $h_X(t)$ is the norm in $\bar X$ of the dilation operator $D_t(f)(s)=f(s/t)$, $0<t<\infty$. It is easy to see that if $X$ is a Banach space, then
\begin{equation}\label{funf}
\varphi_X(t)\varphi_{X'}(t)\approx t.
\end{equation} 
 
 From formula (\ref{dist2}) we can also obtain results concerning the boundedness of operators on r.i.\ spaces with weights. To this end, given an RIBFS $X$ let us recall that  the Hardy-Littlewood maximal operator is bounded in $X$ if and only if $p_X>1$ (see \cite[Theorem 5.17]{BS}). Also, recall that the Marcinkiewicz space is defined by
 $$
 \Vert f\Vert _{M_X}=\sup_{t>0} f^{**}(t) \varphi_X(t).
 $$
 
It is proved  in \cite{ms} that if $X$ is  a quasi-Banach rearrangement-invariant space with lower Boyd index $p_X$  and upper Boyd index $q_X$, and we define the operators
 
$$
S_p f (t) = \bigg( {1\over t} \int_0^t f^*(s)^{p}ds\bigg)^{1/p}
$$
and
$$
S^*_q f (t) = 
t^{-1/q}\bigg ( \int_t^\infty f^*(s)^q ds\bigg)^{1/q},
$$
then  $S:X\rightarrow X$  if and only 
if $p_X > p$ and  $S^*_q: X \rightarrow X$   if and only if $q_X < q$. 
 
 Let us now define $f^*_w$ to be the nonincreasing rearrangement of $f$ with respect to the measure $w(x)\,dx$, and 
 $$
 X(w)=\{ f; \ f^*_w\in \bar X\}.
 $$
 Then the following result follows (see \cite{CGMP} for related questions):
 
 \begin{theorem} 
  If $f$ and $g$ satisfy (\ref{debil1}) for every $w\in A_{p_0}$, then, for every  RIBFS $X$  such that $1<p_X\le q_X <\infty$, and every $w\in A_{p_X}$: 
 $$
\Vert g\Vert _{M_X(w)}\lesssim \Vert f\Vert _{X(w)},
$$
with constant depending on $\Vert w\Vert _{A_{p_X}}$. 
 
 \end{theorem}
 
 \begin{proof}   Since $w\in A_{p_X}$, we have that $w\in A_p$ for some $p<p_X$ and  hence it is known (see \cite{CRS}) that
 $$
 (M f)^*_w(t) \lesssim \bigg({1\over t} \int_0^t f^*_w(s)^{p} ds\bigg)^{1/p}.
 $$
 Now, using the result in \cite{ms} mentioned above, we have that since $p<p_X$
 $$
 \Vert Mf\Vert _{X(w)}\le \Vert S_p f^*_w\Vert _{\bar X}\le \Vert f^*_w\Vert _{\bar X} =\Vert f\Vert _{X(w)};
 $$
 that is,
 $$
 M:X(w)\longrightarrow X(w).
 $$
 By taking $\delta$ sufficiently near $1$ we have the same boundedness for $M_\delta$, and hence
 $$
 \lambda_{M_\delta f(y)}^w\lesssim \varphi_X^{-1}\bigg({\Vert f\Vert _{X(w)}\over y}\bigg). 
 $$
 
To estimate the second term in (\ref{dist2}) we use duality to conclude  
$$
\int_{\mathbb R^n} f(x) M_{\mu} (w \chi_{\{ g>y\}})(x) dx\le \Vert f\Vert _{X(w)} \Vert  w^{-1} M_{\mu} (w \chi_{\{ g>y\}})\Vert _{X'(w)}.
$$

Now we claim that 
\begin{equation}\label{mmu}
\Vert  w^{-1} M_{\mu} (w \chi_{\{ g>y\}})\Vert _{X'(w)}\lesssim \varphi_{X'}(\lambda_g^w(y)),
\end{equation}
and hence, using (\ref{funf}), we get
\begin{eqnarray*}
\lambda_g^w(y) &\lesssim&  \varphi_X^{-1}\bigg({\Vert f\Vert _{X(w)}\over y}\bigg) +{1\over y}  \Vert f\Vert _{X(w)} \varphi_{X'}(\lambda_g^w(y))
\\
&\lesssim&
 \varphi_X^{-1}\bigg({\Vert f\Vert _{X(w)}\over y}\bigg) +{1\over y}  \Vert f\Vert _{X(w)} {\lambda_g^w(y)\over \varphi_{X}(\lambda_g^w(y))},
\end{eqnarray*}
from which it follows that 
$$
\Vert g\Vert _{M_X(w)}=\sup_{y>0} y\varphi_X(\lambda_g^w(y))\lesssim \Vert f\Vert _{X(w)},
$$
as we wanted to see. 

To finish the proof we need to prove the claim (\ref{mmu}). This will be a consequence of the following more general theorem. 

  \end{proof}
  
  \begin{theorem}\label{bien}   For every RIBFS  $X$ such that $1<p_X\le q_X<\infty $ and every $w\in A_{p_X}$, there exists   $0<\mu<1 $, for which  the operator
  $ T(g)=w^{-1} M_\mu(w g)$ satisfies 
$$
  T:X'(w) \longrightarrow X'(w).
  $$

   \end{theorem}

\begin{proof}  For simplicity, we shall give the proof for $M$ instead of $M_\mu$, but everything can be immediately checked for $M_\mu$ (where $\mu$ is chosen appropriately).  

First of all we observe that if $w\in A_s$, then $w^{1-s'}\in A_{s'}$ and hence  we have that
$T:L^{s'}(w)\longrightarrow L^{s'}(w)$. 

Now, since $w\in A_{p_X}$,   if $q_0'>q_X$,  then $w\in A_{q'_0}$ and we can also take $q_1'<p_X$ such that $w\in A_{q_1'}$.

Hence, we have that
$$
T:L^{q_0}(w) \longrightarrow L^{q_0}(w),
$$
and also
$$
T:L^{ q_1}(w) \longrightarrow L^{q_1}(w).
$$
Then
$$
K(Tf, t; L^{q_0}(w), L^{q_1}(w)) \lesssim K(f, t; L^{q_0}(w), L^{q_1}(w)),
$$
with $K$ the Peetre $K$-functional (see \cite[Definition 1.1]{BS}). Now, it is known that 
\begin{eqnarray*}
K(f, t; L^{q_0}(w), L^{q_1}(w))&\approx& \bigg(\int_0^{t^{q_1 q_0\over q_1-q_0}} f^*_w(s)^{q_0} ds\bigg)^{1/q_0}+
t\bigg(\int_{t^{q_1 q_0\over q_1-q_0}}^\infty f^*_w(s)^{q_1} {ds}\bigg)^{1/q_1}
\\
&\approx&
 t^{{q_1 \over q_1-q_0}}\Big(S_{q_0} f^*_w(t^{{q_1 q_0\over q_1-q_0}})+ S^*_{ q_1} f^*_w (t^{{q_1 q_0\over q_1-q_0}})\Big),
\end{eqnarray*}
  (this is a consequence of  \cite[Theorem 2.1]{BS}). Now, observe that, for every decreasing  function $h$
 $$
 h (t)\le
 S_{q_0} h(t)+ S^*_{ q_1} h(t),
$$
and hence, since $p_{X'}>q_0$ and $q_1>q_{X'}=p_X'$
\begin{eqnarray*}
\Vert Tf\Vert _{X'(w)}&=& \Vert (Tf)^*_w\Vert _{X'}\le \Vert S_{q_0} (Tf)^*_w+ S^*_{q_1}(Tf)^*_w\Vert _{X'}
\\
&\le&  \Vert S_{q_0} f^*_w+S^*_{ q_1}f^*_w\Vert _{X'}\lesssim \Vert f^*_w\Vert _{X'}= \Vert f\Vert _{X'(w)},
\end{eqnarray*}
as we wanted to see. 
\end{proof}

\begin{remark} Observe that from the above proof we have, in fact,  that if $u\in A_{p_0}$ and
$p_0<p_X\le q_X<p_1$, then $X(u)$ is an interpolation space between $L^{p_0}(u)$ and $L^{p_1}(u)$
and hence, we could apply Rubio de Francia's extrapolation theorem to deduce the following stronger result (see also \cite{cmpn:cmpn}). 

\end{remark}

\medskip

\begin{theorem} If $T$ is a sublinear operator such that
$$
T:L^{p_0}(w)\longrightarrow L^{p_0, \infty}(w),
$$
for every $w\in A_{p_0}$, then for every RIBFS $X$  such that $1< p_X\le q_X<\infty$ and every $w\in A_{p_X}$:
$$
T:X(w)\longrightarrow X(w).
$$
\end{theorem}

\bigskip

\subsection{Boundedness of operators in Rearrangement Invariant Spaces}

\

\ 

In this Section we consider the case $u=1$ and the goal is to prove the boundedness on r.i.\ spaces for an operator $T$ such that
$$
T:L^{p_0}(w) \longrightarrow L^{p_0, \infty}(w),
$$
for every $w\in A_{p_0}$. Observe that, in this case, using (\ref{distF}) and the fact that
$$
Mf^*(t)\approx {1\over t} \int_0^t f^*(s)\, ds,
$$
we have  that, for every $s>1$, 
\begin{equation}\label{rii}
\lambda_g(y)\lesssim  \lambda_{Mf}(y) + {1\over y^{s}} \int_0^\infty  f^*(t)^{s}\min\Big( {\lambda_g(y)\over t}, 1\Big)^{\mu} dt,
\end{equation}
and also, by  (\ref{dist2})
$$
\lambda_g(y)\lesssim  \lambda_{M_\delta f}(y) + {1\over y} \int_0^\infty  f^*(t) \min\Big( {\lambda_g(y)\over t}, 1\Big)^{\mu} dt.
$$

The parameter $s$ can be taken bigger than zero if we have the boundedness for every $w\in A_\infty$.

\begin{theorem}\label{324} Let $X$ be a quasi-Banach r.i.\ space satisfying that

\noindent
(i) $\varphi_{X'}(t) \lesssim {t\over  \varphi_{X}(t) }$.

\noindent
(ii)  There exists $0<\delta<1$ such that $M_{\delta}:X\rightarrow M^\infty(X)$, where 
$$
\Vert f\Vert _{M^\infty(X)}=\sup_{t>0} t \varphi_X (\lambda_f(t))=\sup_{s>0} \varphi_X(s) f^*(s).
$$
 
 \noindent
(iii)  There exists $0<\mu<1$ such that
$$
\bigg\Vert{1\over (1+{\cdot \over s})^\mu}\bigg\Vert_{\bar {X'}} \lesssim \varphi_{X'}(s). 
$$
Then, if $(f,g)$ satisfies (\ref{debil1}) for some $1\le p_0<\infty$, we have that
$$
\Vert g\Vert _{M^\infty(X)}\lesssim \Vert f\Vert _{X}.
$$
\end{theorem}

\begin{proof} By conditions (iii) and (i) we have that (the norms are taken with respect to the $t$ variable):
\begin{eqnarray*}
 \Big\Vert\min\Big({s\over t}, 1\Big)^\mu \Big\Vert_{\bar X'}&= & \Big\Vert \chi_{(0,s)}(t)+ \Big({s\over t}\Big)^\mu \chi_{(s, \infty}(t)\Big\Vert_{\bar X'}
\\
&\lesssim&  \varphi_{X'}(s)+ \Big\Vert\Big({s\over t}\Big)^\mu \chi_{(s, \infty}(t)\Big\Vert_{\bar X'}
\lesssim\varphi_{X'}(s)\lesssim {s\over  \varphi_{X}(s) },
\end{eqnarray*}
and  by (ii)  $\lambda_{M_\delta f}(y)\lesssim \varphi_X^{-1}\bigg({\Vert f\Vert _X\over y}\bigg)$. Hence using formula (\ref{dist2})  we obtain that 
$$
\lambda_g(y) \lesssim \varphi_X^{-1}\bigg({\Vert f\Vert _X\over y}\bigg)+ {\Vert f\Vert _X\over y}{\lambda_g(y)\over  \varphi_{X}(\lambda_g(y))},
$$ 
from which the result follows. 
\end{proof}

\medskip

\begin{remarks}\ 

\begin{enumerate}

\item[{(a)}] Conditions (i) and (iii) could be replaced by:

\medskip
\noindent
(i') There exists $0<\mu<1$ such that
$$
\Big\Vert\min\Big({s\over t}, 1\Big)^\mu \Big\Vert_{X'}\lesssim {t\over  \varphi_{X}(t) }. 
$$

\item[{(b)}] Condition (iii) implies that $(1+t)^{-\mu}\in \bar {X'}$, which is equivalent to the embedding $(L^\infty\cap L^{1/\mu, \infty})\subset X'$. 

\item[{(c)}] Condition (i) always holds if $X$ is Banach (see (\ref{funf})). 
\end{enumerate}
\end{remarks}

Let us now consider  the particular case of weighted Lorentz spaces (see \cite{CRS}): 
 $$
 \Lambda^p(w)=\bigg\{ f; \Vert f\Vert _{\Lambda^p(w)}= \bigg(\int_0^\infty f^*(t)^p w(t) dt\bigg)^{1/p}<\infty\bigg\},
 $$
 with $0<p<\infty$. In this case 
 $$
 M^\infty (X)=\Lambda^{p,\infty}(w)=\bigg\{ f; \Vert f\Vert _{\Lambda^{p,\infty}(w)}= \sup_{t>0} f^*(t) W^{1/p}(t) <\infty\bigg\},
 $$
 with $W(t)=\int_0^t w(s) ds$.
 
 \, 
 
 The boundedness of the Hardy-Littlewood maximal operator on the weighted Lorentz spaces was first characterized by M.A.\ Ari\~{n}o and B.\ Muckenhoupt in \cite{ArMu} by the condition that $w\in B_p$; that is,
 $$
 r^p \int_r^\infty \frac{w(t)}{t^p} dt \lesssim \int_0^r w(t) dt. 
 $$
 
 In fact, the result holds true for every $p>0$. 
 
 \, 
 
 This class of weights has many properties in common with the class $A_p$. In particular, if $w\in B_p$, there exists $\varepsilon>0$ such that $w\in B_{p-\varepsilon}$. 
 
 Another important condition for us proved in \cite{s1:s1} is that, if $p> 1$,  $w\in B_p$ if and only if $\Lambda^p(w)$ is a Banach space. Also, several equivalent characterizations of $B_p$ weights are given in \cite{so:so}.

  \begin{theorem} Let  $T$ be an operator such that 
$$
T:L^{p_0}(u) \longrightarrow L^{p_0, \infty}(u),
$$
for every $u\in A_{p_0}$, with constant depending on $\Vert u\Vert _{A_{p_0}}$.  If 

\begin{enumerate}
\item[{(a)}] $1<p<\infty$,  $w\in B_p$, $W(\infty)=\infty$ and for some $0<\mu<1$, 
 \begin{equation}\label{cond}
 \bigg(  \int_s^\infty \bigg({t^{1-\mu} \over W(t)} \bigg)^{p'} w(t) dt\bigg)^{1/p'}\lesssim {s^{1-\mu}\over W^{1/p}(s)}, 
  \end{equation}
  \end{enumerate}
  or 
  \begin{enumerate}
  \item[{(b)}]  $0<p\le 1$ and $w$ satisfies that, for some $0<\delta<1$, $W^{1/p}(t)/t^{\delta}$ is quasi-decreasing and for some $0<\mu<1$, $W^{1/p}(t)/t^{1-\mu}$ is quasi-increasing,
    \end{enumerate}
    then,
$$
T:\Lambda^{p}(w) \longrightarrow \Lambda^{p,\infty}(w) 
$$
is bounded.

  \end{theorem}
  
  \begin{proof}   
  
  \noindent
  (a) Since $w\in B_p$, we have that there exists ${1\over p}<\delta<1$ such that $w\in B_{\delta p}$ and hence  $M_\delta: \Lambda^p(w)\longrightarrow  \Lambda^{p,\infty}(w)$ (see \cite{CRS}); that is, condition (ii) on Theorem~\ref{324} is satisfied. Now, since $w\in B_p$, the space $X=\Lambda^p(w)$ is Banach and hence also condition (i) is satisfied. 
  
  To study condition (iii) we need to use Sawyer's formula \cite{s1:s1} and since $W(\infty)=\infty$, we have that 
  \begin{eqnarray*}
\bigg\Vert{1\over (1+{t\over s})^\mu}\bigg\Vert_{X'}&=&\sup_{f \downarrow} {\int_0^\infty {f(t)\over (1+{t\over s})^\mu}dt\over \bigg(\int_0^\infty f^p(s) w(s) ds\bigg)^{1/p}}
\\
& \approx&  \bigg(\int_0^\infty \bigg({s\Big[(1+{t\over s})^{1-\mu}-1\Big]\over W(t)}\bigg)^{p'} w(t) dt \bigg)^{1/p'}.
   \end{eqnarray*}
 Now, since $w\in B_p$ (see \cite{so:so})  we have that 
 $$
 \int_0^s \bigg({s\Big[(1+{t\over s})^{1-\mu}-1\Big]\over W(t)}\bigg)^{p'} w(t) dt \le 
\int_0^s \bigg({t \over W(t)}\bigg)^{p'} w(t) dt \lesssim \bigg({s\over W^{1/p}(s)}\bigg)^{p'}
$$
  and by (\ref{cond}) 
 \begin{eqnarray*}
\bigg(\int_s^\infty \bigg({s\Big[(1+{t\over s})^{1-\mu}-1\Big]\over W(t)}\bigg)^{p'} w(t) dt \bigg)^{1/p'} &\lesssim& s^\mu \bigg(\int_s^\infty \bigg({t^{1-\mu} \over W(t)}\bigg)^{p'} w(t) dt \bigg)^{1/p'}
\\
&\lesssim&
{s\over W^{1/p}(s)}.
 \end{eqnarray*}
Therefore, condition (iii) is satisfied. 

\ 

\noindent
(b) Set now $0<p\le 1$. Since by hypothesis ${W^{1/p}(t)\over t}$ is decreasing  we have (see \cite{cs2:cs2}) that 
$$
\varphi_{X'}(t)=\Vert \chi_{(0,t)}\Vert _{\bar X'}=\sup_{r>0}{\min(r,t)\over W(r)^{1/p}}\approx {t\over W(t)^{1/p}},
  $$

  \noindent
  and hence condition (i) in the previous theorem holds. 
  Now, since $0<\delta<1$ and $W^{1/p}(t)/t^{\delta}$ is quasi-decreasing,  it is known that condition (ii) holds  (see \cite{CRS}). So, it remains to prove condition (iii): Let us consider $\mu$ such that $W^{1/p} (t) t^{\mu-1}$ is quasi-increasing. Then (see \cite{cs2:cs2}):
   \begin{eqnarray*}
\bigg\Vert{1\over (1+{t\over s})^\mu}\bigg\Vert_{\bar X'}&=&\sup_{f \downarrow} {\int_0^\infty {f(t)\over (1+{t\over s})^\mu}dt\over \bigg(\int_0^\infty f^p(s) w(s) ds\bigg)^{1/p}}
\\
& \approx& \sup_{r>0} s{(1+{r\over s})^{1-\mu}-1\over W^{1/p}(r)} \lesssim {s\over W^{1/p}(s)},\end{eqnarray*}
 as we wanted to see.

    \end{proof}
    
    \medskip
    
    \begin{remarks} 
    
  \ 
    
    \noindent
    (i) Given $0<\mu\le 1$, set
    $$
    Q_\mu f(t)=\int_t^\infty {f(s)\over s^\mu} ds
    $$
    Then, it is known (see \cite{cs:cs}) that, for $1<p<\infty$, 
    $$
    Q_\mu:L^p_{\rm dec}(w)\longrightarrow L^{p,\infty}(w)
    $$
    if and only if (\ref{cond}) holds if $\mu<1$ and if $\mu=1$
    
 $$
 \bigg(\int_s^\infty \bigg({\log {t\over s}\over W(t)}\bigg)^{p'} w(s) ds \bigg)^{1/p'} \lesssim {1\over W^{1/p}(s)}.
 $$
  It turns out that this last condition together with $w\in B_p$ is necessary and sufficient for the boundedness on $\Lambda^p(w)$ of the Hilbert transform (\cite{s1:s1}).  
  
  \noindent
  (ii) If $0<p\le 1$, the condition that there exists $\mu$ such that $W^{1/p}(t)t^{\mu-1}$ is quasi-increasing is equivalent 
  to the fact that the operator
  $$
  R_\mu f(t)={1\over t^{1-\mu}}\int_t^\infty {f(s)\over s^\mu}ds
  $$
 satisfies
  $$
  R_\mu:L^{p}_{\rm dec}(w) \rightarrow L^{p,\infty}(w). 
  $$
    To see this, observe that  $R_\mu f$ is decreasing and hence, for every $t>0$,
    \begin{eqnarray*}
    \sup_{f\downarrow }{W^{1/p}(t) t^{\mu-1} \int_t^\infty {f(s)\over s^\mu} ds\over \bigg(\int_0^\infty f^p(t) w(t) dt\bigg)^{1/p}}&=&\sup_{r>t }{W^{1/p}(t) t^{\mu-1} \int_t^r s^{-\mu} ds\over W^{1/p}(r)}\\
    &\approx &
    {W^{1/p}(t)\over t^{1-\mu}} \sup_{r>t }{r^{1-\mu}-t^{1-\mu}\over W^{1/p}(r)},
       \end{eqnarray*}
    which is finite by the hypothesis assumed  on $w$. 
    
    \end{remarks}

Using now formula (\ref{rii})  with $s>1$ in the case of  $A_p$  weights  and $s>0$  in the case of $A_\infty$ weights, we can also prove the following: 
    
    \begin{theorem} Let  $T$ be an operator such that 
$$
T:L^{p_0}(u) \longrightarrow L^{p_0, \infty}(u)
$$
for every $u\in A_{p_0}$ with constant depending on $\Vert u\Vert _{A_{p_0}}$.

\noindent
(a) If    for some $0<\mu<1$ and some $0<r<p$, 
 \begin{equation}\label{normmin}
 \Big\Vert \min\Big(1, {s\over t}\Big)^\mu\Big\Vert _{\Lambda^r(w)^*}\lesssim s^{1/r'}
    \end{equation}
    and $w\in B_p$   if $1<p<\infty$ or   $W^{1/p}(t)/t $ is quasi-decreasing   if $p\le 1$, then 
     $$
T:\Lambda^{p}(w) \longrightarrow \Lambda^{p,\infty}(w) 
$$
is bounded. 

\noindent
  (b)   If  the hypothesis holds for every $u\in A_\infty$,  the parameter $r$ in (a) can be taken $0<r<\infty$ to obtain the same conclusion.

  \end{theorem}
  
  \begin{proof} Let us take $s={p/ r}$ in formula (\ref{rii}).  Then
  \begin{eqnarray*}
\lambda_g(y)&\lesssim&  \lambda_{Mf}(y) + {1\over y^{s}} \int_0^\infty  f^*(t)^{s}\min\Big( {\lambda_g(y)\over t}, 1\Big)^{\mu} dt
\\
&\lesssim&
{\Vert f\Vert ^p_{\Lambda^p(w)}\over y^p}+{1\over y^{s}} \Vert f\Vert ^s_{\Lambda^p(w)} \Big\Vert \min\Big(1, {\lambda_g(y)\over t}\Big)^\mu\Big\Vert _{\Lambda^r(w)^*}
  \end{eqnarray*}
  and by  hypothesis we have that 
   \begin{eqnarray*}
& &\lambda_g(y)\lesssim
{\Vert f\Vert ^p_{\Lambda^p(w)}\over y^p}+{1\over y^{s}} \Vert f\Vert ^s_{\Lambda^p(w)}\lambda_g(y)^{{1\over r'}}  \end{eqnarray*}
from which the result follows. 
  
  \end{proof}

 \begin{remark} Condition (\ref{normmin}) can be estimated as follows (see \cite{cpss,s1:s1,cs:cs}).   Let $0<s, r<\infty$, $0<\mu<1$.

\noindent
(i) If $r>1$, 
\begin{eqnarray*}
\Big\Vert \min\Big(1, {s\over t}\Big)^\mu\Big\Vert _{\Lambda^r(w)^*}& \approx & \bigg( \int_0^s \bigg({t\over W(t)}\bigg)^{r'}w(t) dt\\
 &\ &\qquad + s^{\mu r'} \int_s^\infty  \bigg( \int_0^s \bigg({t^{1-\mu}\over W(t)}\bigg)w(t) dt \bigg)^{1/r'}.
\end{eqnarray*}

\noindent
(ii) If $r\le 1$, 
$$
\Big \Vert\min\Big(1, {s\over t}\Big)^\mu\Big \Vert _{\Lambda^r(w)^*}\approx \max \bigg( \sup_{0<r\le s} {r\over W^{1/p}(r)},  \sup_{s<r} {s^\mu r^{1-\mu}\over W^{1/p}(r)}\bigg).
$$
     
\end{remark}

 \section{The strong-type case}\label{strong}
 
 In the usual cases treated in the literature, the starting hypothesis is that 
 \begin{equation}\label{strong2}
\Vert g\Vert _{L^{p_0}(w)}\le C_u \Vert f\Vert _{L^{p_0}(w)},
 \end{equation}
 for every $w\in A_{p_0}$. We shall give in this Section very easy proofs of the classical Rubio de Francia's extrapolation result both in the linear and in the multilinear case (see \cite{R} and \cite{gm:gm} respectively):

 \begin{theorem} If (\ref{strong2})  holds, for some $p_0>1$ and all $w\in A_{p_o}$, then  for every $p>1$ and every $w\in A_{p}$
 $$
\Vert g\Vert _{L^{p}(w)}\le C_w \Vert f\Vert _{L^{p}(w)}.
 $$
 \end{theorem}
 
 \begin{proof} Let $w\in A_p$. Then 

 \begin{eqnarray*}
  \int_{\mathbb R^n} |g(x)|^p w(x) dx
&\le&   \int_{\{M_\delta f>g\}} |M_\delta f(x)|^p w(x) dx+ 
  \int_{ \{ M_\delta f\le g\}} g(x)^p w(x) dx
  \\
 &\le&   \int_{\mathbb R^n} |M_\delta f(x)|^p w(x) dx+ 
  \int_{\mathbb R^n} g(x)^p \bigg({g(x)\over M_\delta f(x)}\bigg)^{p_0-1}w(x) dx\\
  & =&I+II.
 \end{eqnarray*}
 The bound for $I$ is clear. For the other term we have 
$$
 II\le  \int_{\mathbb R^n } g(x)^{p_0} (M_\delta f)(x)^{1-p_0} M_\mu(g^{p-1} w)(x)dx,
 $$
 and since $(M_\delta f)^{1-p_0} M_\mu(g^{p-1} w)\in A_{p_0}$  we can apply the hypothesis to get that 
  \begin{eqnarray*}
 II&\lesssim&   \int_{\mathbb R^n } f(x)^{p_0} (M_\delta f)(x)^{1-p_0} M_\mu( g^{p-1} w)(x)dx
 \\
 &\lesssim& 
  \int_{\mathbb R^n } f(x) M_\mu( g^{p-1}w )(x)dx\le \Vert f\Vert _{L^p(w)} \Vert M_\mu( g^{p-1} w)\Vert _{L^{p'}(w^{1-p'})}
  \\
  &\lesssim& \Vert f\Vert _{L^p(w)} \Vert g\Vert ^{p-1}_{L^p(w)},
 \end{eqnarray*}
where $\mu$ is chosen sufficiently close to 1. Consequently,
 $$
 \Vert g\Vert ^p_{L^p(w)}\lesssim \Vert f\Vert ^p_{L^p(w)}+ \Vert f\Vert _{L^p(w)} \Vert g\Vert^{p-1}_{L^p(w)},
 $$
 from which the result follows. 
 \end{proof}
 
 Let us see now the multilinear version. 
 
  \begin{theorem} Let $1/p_1+1/p_2=1/p$ ($1<p_1,p_2)$, and let  $(f_1, f_2, g)$ be such that for every $u_j\in A_{p_j}$ and $u=u_1^{p/p_1}u_2^{p/p_2}$,
 $$
\Vert g\Vert _{L^{p}(u)}\le C_{u_1, u_2} \Vert f_1\Vert _{L^{p_1}(u_1)}\Vert f_2\Vert _{L^{p_2}(u_2)}.
 $$
 Then, for every $w_j\in A_{q_j}$ and $w=w_1^{q/q_1}w_2^{q/q_2}$, with $1/q_1+1/q_2=1/q$ ($1<q_1,q_2)$,
 $$
\Vert g\Vert _{L^{q}(u)}\le C_{w_1, w_2} \Vert f_1\Vert _{L^{q_1}(w_1)}\Vert f_2\Vert _{L^{q_2}(w_2)}. 
 $$
 \end{theorem}
 
 \begin{proof} The proof is an extension of the above proof in the linear case together with the idea developed in Theorem \ref{rfml} for the weak case. 
 
  Let us start assuming that  $p_1=p_2=2p$ and let $w_j\in A_{q_j}$ and $w=w_1^{q/q_1}w_2^{q/q_2}$.

 \begin{eqnarray*}
 & & \int_{\mathbb R^n} |g(x)|^{q} w(x) dx
 \\
 &\le&   \int_{\{M_\delta f_1 M_\delta f_2 >g\}} (M_\delta f_1 (x) M_\delta f_2 (x)   )^{q} w(x) dx+ 
  \int_{ \{M_\delta f_1 M_\delta f_2 \le g\}} g(x)^{q} w(x) dx
  \\
 &\le&  \int_{\mathbb R^n} (M_\delta f_1 (x) M_\delta f_2 (x)   )^{q} w(x) dx+ 
  \int_{\mathbb R^n} g(x)^{q}\bigg({g(x)\over M_\delta f_1 (x) M_\delta f_2 (x)}\bigg)^{{p_1-1\over 2}}w(x) dx 
  \\
  &=&I+II
 \end{eqnarray*}
 Now, 
$$
 II\le  \int_{\mathbb R^n } g^{p} (M_\delta f_1)^{{1-p_1\over 2}} (M_\delta f_2)^{{1-p_2\over 2}}M_\mu(g^{{q\over q'_1}} w_1^{\beta_1}w_2^{\beta_2})^{{1/2}}M_\mu( g^{{q\over q'_2}} w_1^{\gamma_1}w_2^{\gamma_2})^{{1/2}}dx,
 $$
 and since $(M_\delta f_1)^{1-p_1} M_\mu(g^{{q\over q'_1}} w_1^{\beta_1}w_2^{\beta_2})\in A_{p_1}$ and  $(M_\delta f_2)^{1-p_2} M_\mu(g^{{q\over q'_2}} w_1^{\gamma_1}w_2^{\gamma_2})\in A_{p_2}$ we can apply the hypothesis to get that, choosing the parameters $\beta_j$ and $\gamma_j$ as in Theorem
 \ref{rfml}, 
  \begin{eqnarray*}
 II&\le&   \bigg(\int_{\mathbb R^n } f_1^{p_1}(x) M_\delta f_1(x)^{1-p_1} M_\mu(g^{{q\over q'_1}} w_1^{\beta_1}w_2^{\beta_2})(x)dx\bigg)^{1/2}
 \\
 &&\qquad \times
  \bigg(\int_{\mathbb R^n } f_2^{p_2}(x) M_\delta f_2(x)^{1-p_2} M_\mu(g^{{q\over q'_2}} 
  w_1^{\gamma_1}w_2^{\gamma_2})(x)dx\bigg)^{1/2}
  \\
  &\le&   \bigg(\int_{\mathbb R^n } f_1(x)  M_\mu(g^{{q\over q'_1}} w_1^{\beta_1}w_2^{\beta_2})(x)dx\bigg)^{1/2}\!
  \bigg(\int_{\mathbb R^n } f_2(x)M_\mu(g^{{q\over q'_2}} 
  w_1^{\gamma_1}w_2^{\gamma_2})(x)dx\bigg)^{1/2}
 \\
 &\le&  \Vert f_1\Vert ^{1/2}_{L^{q_1}(w_1)}\Vert f_2\Vert ^{1/2}_{L^{q_2}(w_2)} \Vert g\Vert ^{{q\over 2q'_1}+ {q\over 2q'_2}}_{L^q(w)}.
 \end{eqnarray*}
 Consequently,
 $$
 \Vert g\Vert ^q_{L^q(w)}\lesssim \Vert f_1\Vert ^{q}_{L^{q_1}(w_1)}\Vert f_2\Vert ^{q}_{L^{q_2}(w_2)}+ 
  \Vert f_1\Vert ^{1/2}_{L^{q_1}(w_1)}\Vert f_2\Vert ^{1/2}_{L^{q_2}(w_2)} \Vert g\Vert ^{{q\over 2q'_1}+ {q\over 2q'_2}}_{L^q(w)},
  $$
 from which the boundedness we are looking for  follows. 
 
Now, it remains to extrapolate from a general $(p_1, p_2)$ to a point in the diagonal as it was done in the weak case. 
 
We can assume, without loss of generality, that  $p_1\le p_2$ and let us extrapolate to the case $q_1=q_2=p_2$. Let $w_j\in A_{p_2}$ and $w=w_1^{1/2}w_2^{1/2}$.

 \begin{eqnarray*}
 & & \int_{\mathbb R^n} |g(x)|^{p_2/2} w(x) dx
 \le    \int_{\{M_\delta f_1 M_\delta f_2 >g\}} (M_\delta f_1 (x) M_\delta f_2 (x)   )^{p_2/2} w(x) dx
  \\
& & \qquad + 
  \int_{ \{M_\delta f_1 M_\delta f_2 \le g\}} g(x)^{p_2/2} w(x) dx
  \\
 &\le&  \int_{\mathbb R^n} (M_\delta f_1 (x) M_\delta f_2 (x)   )^{p_2/2} w(x) dx\\
 & & \qquad+ 
  \int_{\mathbb R^n} g(x)^{p_2/2}\bigg({g(x)\over M_\delta f_1 (x) M_\delta f_2 (x)}\bigg)^{{p\over p'_1}}w(x) dx 
  \\
  &=&I+II.
 \end{eqnarray*}
 Now, let $v\in A_1$ and $s=(1-p_2)\Big(1-{p'_2\over p'_1}\Big)$ as in Theorem \ref{multi}. Then, 
 \begin{eqnarray*}
 II&\le&  \int_{\mathbb R^n } g^{p} (M_\delta f_1)^{{(1-p_1)p\over p_1}} \Big((M_\delta f_2)^{{(1-p_2) p'_2\over p'_1}}\Big)^{p/p_2}v^{{sp\over p_2}}\\
 &&\qquad\times M_\mu(g^{{p_2\over 2 p'_2}} w_1^{\beta_1}w_2^{\beta_2})^{{p\over p_1}}M_\mu(v^{-s} g^{{p_2\over 2 p'_1}} w_1^{\gamma_1}w_2^{\gamma_2})^{{p\over p_2}}dx
  \end{eqnarray*}
 and since 
 $$
 (M_\delta f_1)^{1-p_1} M_\mu(g^{{p_2\over 2 p'_2}} w_1^{\beta_1}w_2^{\beta_2})\in A_{p_1}
 $$  
 and 
 $$
 (M_\delta f_2)^{{(1-p_2)}{p'_2\over p'_1}}v^{{s}}  M_\mu(v^{-s} g^{{p_2\over 2 p'_1}} w_1^{\gamma_1}w_2^{\gamma_2})\in A_{p_2},
 $$
 we can apply the hypothesis to get that 
  \begin{eqnarray*}
 II&\le&   \bigg(\int_{\mathbb R^n } f_1(x)^{p_1} (M_\delta f_1)^{1-p_1}(x) M_\mu(g^{{p_2\over 2 p'_2}} w_1^{\beta_1}w_2^{\beta_2})(x)dx\bigg)^{p/p_1}
 \\
 & & \qquad\times
 \bigg( \int_{\mathbb R^n } f_2(x)^{p_2} (M_\delta f_2)^{{(1-p_2)}{p'_2\over p'_1}}v^{{s}}(x)  M_\mu(v^{-s} g^{{p_2\over 2 p'_1}} w_1^{\gamma_1}w_2^{\gamma_2})(x) dx\bigg)^{p/p_2}  \\
 &\le&   \bigg(\int_{\mathbb R^n } f_1(x) M_\mu(g^{{p_2\over 2 p'_2}} w_1^{\beta_1}w_2^{\beta_2})(x)dx\bigg)^{p/p_1}
 \\
 & & \qquad\times
 \bigg( \int_{\mathbb R^n } f_2(x)^{p_2/p_1}v^{{s}}(x)  M_\mu(v^{-s} g^{{p_2\over 2 p'_1}} w_1^{\gamma_1}w_2^{\gamma_2})(x) dx\bigg)^{p/p_2}=II_1 \times II_2.
 \end{eqnarray*}
 To estimate the first term we proceed as usual by duality and we get that if 
 $$
\beta_1 p_2' +1-p_2'={1\over 2},\qquad \beta_2 p'_2={1\over 2},
$$
then
 \begin{eqnarray*}
 & &II_1\le \Vert f_1\Vert ^{p\over p_1}_{L^{p_2}(w_1)}  \Vert M_\mu(g^{{p_2\over 2 p'_2}} w_1^{\beta_1}w_2^{\beta_2})\Vert ^{p\over p_1}_{L^{p'_2}(w_1^{1-p'_2})}
\lesssim 
 \Vert f_1\Vert ^{p\over p_1}_{L^{p_2}(w_2)}  \Vert g\Vert ^{pp_2 \over 2p_1 p'_2}_{L^{{p_2\over 2}}(w)},  
  \end{eqnarray*}
  and for the second term, if 
  $$
  \gamma_1 p'_1={1\over 2}, \qquad 1-p'_1+\gamma_2 p'_1={1\over 2},
  $$
  we have,  choosing $v$ as  in Theorem \ref{rfml}, that
  \begin{eqnarray*}
 & & II_2\le  \Vert f_2\Vert ^{p/p_1}_{L^{p_2}(w_2)}  \Vert v^{{s}} M_\mu(v^{-s} g^{{p_2\over 2 p'_1}} w_1^{\gamma_1}w_2^{\gamma_2})\Vert^{p/p_2} _{L^{p'_1}(w_2^{1-p'_1})}
\lesssim 
 \Vert f_2\Vert ^{p/p_1}_{L^{p_2}(w_2)}  \Vert g\Vert ^{pp_2 \over 2p_2 p'_1}_{L^{{p_2\over 2}}(w)}.  
  \end{eqnarray*}

 Consequently,
 $$
 \Vert g\Vert ^{p_2/2}_{L^{p_2}(w)}\lesssim \Vert f_1\Vert ^{p_2/2}_{L^{p_2}(w_1)} \Vert f_2\Vert ^{p_2/2}_{L^{p_2}(w_2)}+
  \Vert f_1\Vert ^{p\over p_1}_{L^{p_2}(w_2)}  \Vert g\Vert ^{pp_2 \over 2p_1 p'_2}_{L^{{p_2\over 2}}(w)}   \Vert f_2\Vert ^{p/p_1}_{L^{p_2}(w_2)}  \Vert g\Vert ^{pp_2 \over 2p_2 p'_1}_{L^{{p_2\over 2}}(w)},  
  $$
 from which the result follows. 
 \end{proof}

\end{document}